\documentclass[a4paper, reqno]{amsart}

\usepackage{marginnote}

\usepackage{anysize}

\usepackage{lineno}
\linenumbers

\usepackage[T1]{fontenc}
\usepackage[style=french]{csquotes}

\usepackage{amssymb,latexsym}
\usepackage{amsmath,amsthm}
\usepackage{amsfonts,mathrsfs}
\usepackage{mathtools}

\usepackage{stmaryrd}

\usepackage{mathptmx} 

\usepackage{halloweenmath}

\usepackage{multicol}

\usepackage{tikz-cd}

\usepackage{todonotes}

\usepackage[dvipsnames]{xcolor}

\usepackage[all]{xy}
\usepackage{graphicx}

\usepackage{xcolor,url}
\usepackage{hyperref}
\hypersetup{
    colorlinks,
    citecolor=blue,
    filecolor=blue,
    linkcolor=black,
    urlcolor=blue
}

\usepackage{enumitem}

\usepackage{ytableau}

\usepackage{listings}
\usepackage{color}

\definecolor{dkgreen}{rgb}{0,0.6,0}
\definecolor{gray}{rgb}{0.5,0.5,0.5}
\definecolor{mauve}{rgb}{0.58,0,0.82}

\usepackage{forest}

\usepackage{tabularray}
\UseTblrLibrary{diagbox}

\usepackage{etoolbox}
\patchcmd{\section}{\normalfont}{\normalfont\color{airforceblue}}{}{}
\patchcmd{\subsection}{\normalfont}{\normalfont\color{airforceblue}}{}{}

\theoremstyle{plain}
\newtheorem{theorem}{Theorem}[section]

\newtheorem{corollary}[theorem]{Corollary}
\newtheorem{lemma}[theorem]{Lemma}

\newtheorem{proposition}[theorem]{Proposition}
\newtheorem{fact}[theorem]{Fact}
\newtheorem{conjecture}[theorem]{Conjecture}
\newtheorem*{theorem*}{Theorem}
\newtheorem*{corollary*}{Corollary}
\newtheorem*{conjecture*}{Conjecture}

\newtheorem*{generalburnsideproblem}{\afbf{The General Burnside Problem}}
\newtheorem*{burnsideproblem}{\afbf{The Burnside Problem}}
\newtheorem*{restrictedburnsideproblem}{\afbf{The Restricted Burnside Problem}}

\theoremstyle{definition}
\newtheorem{definition}[theorem]{Definition}
\newtheorem{example}[theorem]{Example}

\theoremstyle{remark}
\newtheorem{remark}[theorem]{Remark}
\newtheorem{claim}{Claim}
\newtheorem{exercise}[theorem]{Exercise}

\newcommand\C{\mathbb{C}}
\newcommand\Q{\mathbb{Q}}
\newcommand\N{\mathbb{N}}
\newcommand\Z{\mathbb{Z}}
\newcommand\F{\mathbb{F}}

\def\cM{\mathcal{M}}

\def\cO{\mathcal{O}}

\def\Th{\operatorname{Th}}

\def\sgn{\operatorname{sgn}}

\newcommand\m{\mathfrak{m}}
\newcommand\fS{\mathfrak{S}}

\newcommand\LL{\mathscr{L}}
\newcommand\VV{\mathscr{V}}

\newcommand{\End}{\mathrm{End}}

\newcommand{\Id}{\mathrm{Id}}

\newcommand{\alg}{\mathrm{alg}}

\def\seq{\subseteq}

\def\lteq{\trianglelefteq}
\def\rteq{\trianglerighteq}

\newcommand{\set}[1]{\{ {#1} \}}
\newcommand{\vect}[1]{\langle {#1} \rangle}
\newcommand{\abs}[1]{\lvert {#1} \rvert}

\DeclareMathOperator{\Aut}{Aut}

\DeclareMathOperator{\Lie}{Lie}

\DeclareMathOperator{\ad}{ad}
\DeclareMathOperator{\Span}{Span}

\definecolor{airforceblue}{rgb}{0.36, 0.54, 0.66}
\newcommand{\afb}[1]{\textcolor{airforceblue}{#1}}
\newcommand{\ored}[1]{\textcolor{OrangeRed}{#1}}
  
\newcommand{\afbf}[1]{\textcolor{airforceblue}{\textbf{#1}}}

\def\Ind{\setbox0=\hbox{$x$}\kern\wd0\hbox to 0pt{\hss$\mid$\hss}
\lower.9\ht0\hbox to 0pt{\hss$\smile$\hss}\kern\wd0}
\def\Notind{\setbox0=\hbox{$x$}\kern\wd0\hbox to 0pt{\mathchardef
\nn=12854\hss$\nn$\kern1.4\wd0\hss}\hbox to
0pt{\hss$\mid$\hss}\lower.9\ht0 \hbox to 0pt{\hss$\smile$\hss}\kern\wd0}

\def\indi#1{\mathop{\ \ \hbox to 0ex{\hss$\vert^{\hbox to 0ex{$\scriptstyle#1$\hss}}$\hss}
\lower1ex\hbox to 0ex{\hss$\smile$\hss}\ \ }}

\def\nindi#1{\mathop{\ \ \hbox to 0ex{\hss$\!\not{\vert}^{\hbox to 0ex{$\scriptstyle\,#1$\hss}}$\hss}
\lower1ex\hbox to 0ex{\hss$\smile$\hss}\ \ }}

\renewcommand{\models}{\vDash}

\begin{document}

\nolinenumbers

\title{\afbf{A Dubnov-Ivanov-Nagata-Higman Theorem for NIP nilrings}}

\author[C. d'Elb\'{e}e]{Christian d\textquoteright Elb\'ee $^\dagger$}
\address{University of the Basque Country, Department of Mathematics (Leioa) / Institute for Logic, Cognition, Language and Information (Donostia-San Sebastián), Basque Country, Spain}
\email{christian.delbee@ehu.eus}
\urladdr{\href{http://choum.net/\textasciitilde chris/page\textunderscore perso/}{http://choum.net/\textasciitilde chris/page\textunderscore perso/}}

\author{Yanis Davoust}
\address{École normale supérieure de Lyon, Lyon 69007, France}
\email{yanis.davoust@ens-lyon.fr}

\author{Jules Tindzogho Ntsiri}
\address{Université des Sciences et Techniques de Masuku, BP 901, Franceville, Gabon}
\email{jtindzoghontsiri@gmail.com}

\thanks{
$(\dagger)$ Supported by the Basque Government grant reference number IT1913-26, and by the Ramon y Cajal grant RYC2023-042677-I funded by MICIU/AEI/10.13039/501100011033 and by ESF+.
}

\date{\today}

\maketitle

 \begin{center}
 \begin{minipage}{.8\textwidth}
 \section*{Abstract}
We prove that NIP nilrings of bounded nilexponent are nilpotent. We also deduce that NIP $\omega$-categorical $n$-Engel Lie algebras of characteristic $p>n$ are nilpotent, yielding a case of the Lie algebra analogous of Wilson conjecture for $\omega$-categorical groups.
 \end{minipage}
 \end{center}

\section{Introduction}

\subsection{Motivations} This article follows a trend of research at the interface of model theory and algebra related to problems of Burnside-type. It aims at introducing a new set of questions connecting on the one side, the algebraic notions of local and global nilpotency (for groups, Lie rings, associative rings), and on the other side, model-theoretic tameness. The underlying philosophy behind this work can be summarized by the following, whenever nilpotency is a meaningful notion: 
\[
\text{\enquote{locally nilpotent} + \enquote{model-theoretically tame}} \implies \text{\enquote{globally nilpotent}}.\tag{$\dagger$}
\]
A first and important remark is in order. As model theoretic tameness are properties of a theory, our notion of local nilpotency is also to be considered at the level of theories. Therefore, by locally nilpotent, we really mean \textit{uniformly locally nilpotent} in the sense that the nilpotency class of a finitely generated substructure only depends on the number of generators. Section \ref{sec:locniltame} give examples of tame (in fact even dp-minimal) non nilpotent (non-uniformly) locally nilpotent associative algebras.

The notion of local nilpotency has been rather disregarded by model theorists (with the notable exception of \cite{DerakhshanWagner1997}) and it seems that there are no non-nilpotent locally nilpotent tame structures in the current model-theoretic \enquote{map of the universe} \cite{MapOfTheUniverse}. Mekler's construction \cite{Mekler, chernikov2019mekler}, however, yields that nilpotent groups and Lie algebras are ubiquitous in the classification landscape, making essentially \enquote{nilpotency} a notion unaffected by model-theoretic tameness. If $(\dagger)$ holds, structures that are locally nilpotent but not globally nilpotent are to be considered \textit{model-theoretically wild}. 

On the algebraic side, our starting point is the classical result proved independently on each side of the iron wall, by Dubnov and Ivanov \cite{DubnovIvanov1943} and  Nagata and Higman \cite{Nagata1952,Higman1956nagata}: any associative algebra of bounded nilexponent, that is, satisfying an identity of the form $x^n = 0$ over a field of characteristic $0$ or $p>n$ is nilpotent, that is, satisfies an identity of the form $x_1\cdots x_k = 0$. The \textit{Kurosh-Levitzki problem} asks whether associative algebras of bounded nilexponent are (uniformly) locally nilpotent, which where answered positively by Jacobson and Kaplansky \cite{Kaplansky1946, Jacobson1945}. In this paper, we prove that NIP rings of bounded nilexponent are nilpotent, establishing in particular $(\dagger)$ for NIP locally nilpotent rings. 

In fact, $(\dagger)$ for rings stems from the work of model-theorists on the (still incomplete) classification of $\omega$-categorical rings. It starts with the work of Macintyre and Rosenstein \cite{MacintyreRosenstein1976}, where a classification of unital and reduced (i.e. without nilpotent elements) $\omega$-categorical rings is given, in terms of filtered Boolean powers of finite fields à la Arens-Kaplansky. Meanwhile, Baldwin and Rose \cite{baldwinrose1977} study on one side the Jacobson radical $J(R)$ of a given $\omega$-categorical ring $R$, and on the other side the semiprimitive quotient $R/J(R)$. While they obtain little information on $R/J(R)$ except in the case where $R$ is also model-theoretically stable, they prove that $J(R)$ is always a nilring and ask if it is nilpotent, that is, $(\dagger)$ for $\omega$-categorical nilrings. This was established by Cherlin in a tour de force \cite{cherlinnilringI,cherlinnilringII}. Note that for any $\omega$-categorical commutative ring $R$, the quotient $R/J(R)$ is reduced and hence falls into the analysis of Macintyre and Rosenstein, so that we get a complete picture. In the general case, such a nice description is far from complete for $\omega$-categorical rings without extra assumptions. 

As for groups, $(\dagger)$ occurs as one of the most important question in the classification of $\omega$-categorical groups: the Wilson Conjecture. The study of $\omega$-categorical groups was initiated early on by Rosenstein \cite{Rosenstein1973} then pushed forward by J. S. Wilson and Apps \cite{wilson1981, Apps83A,Apps83B, Apps83C}. They established an analysis where characteristically simple $\omega$-categorical $p$-groups serve as building blocks. A complete description of those would follow provided a positive answer to the conjecture: every locally nilpotent $\omega$-categorical $p$-group is nilpotent ; that is, $(\dagger)$ for $\omega$-categorical $p$-groups. Few progress have been made towards a positive answer to this question, notably by Macpherson \cite{Macpherson1988ubiquitousomcatgroups} and Archer and Macpherson \cite{ArcherMacpherson1997}, then later by Krupinski \cite{Krupinski2012}, Wagner and Dobrowolski \cite{DobrowolskiWagner2020}, all under extra model-theoretic assumptions. More recently, the first-named author has been revisiting this question \cite{delbee2024engel3,delbee4engelchar3, delbée2026liemethodscountablycategorical} via a more algebraic approach involving Lie methods in group theory and parametrization by Engel conditions.

The idea of considering $(\dagger)$ across various categories is not merely for the sake of generalization, but is aimed at concrete answers, where results in \enquote{easier} categories help solving questions in \enquote{hard} categories. For instance, one of the main general result proved by Wilson \cite{wilson1981} in order to solve $(\dagger)$ for $\omega$-categorical groups was proved using the answer to $(\dagger)$ of Cherlin \cite{cherlinnilringI} for $\omega$-categorical rings. 


Results of the first-named author towards the Wilson conjecture for Lie algebras and groups are also obtained using trans-category maps principle \cite{delbee2024engel3,delbee4engelchar3, delbée2026liemethodscountablycategorical}: Lie methods in group theory, enveloping associative algebra of a Lie algebra, Lie algebra associated to groups, adjoint group of a nilalgebra, etc... Those interconnections were classically successful in the work of Kostrikin \cite{kostrikinAroundBurnside} and Zelmanov \cite{Zelmanov1990,Zelmanov1991} for the positive solution to the restricted Burnside problem. They also appear in various other Burnside-type problems where \enquote{bounded} exponent conditions relates closely to local or global nilpotency conditions. In a sense, $(\dagger)$ above could be considered as \textit{tame} versions of Burnside-type questions. Note that if $(\dagger)$ indeed holds for NIP groups, that would imply that certain free Burnside groups, namely non-nilpotent Burnside $p$-groups that satisfy the Burnside problem, (for instance Burnside group of exponent $4$ in infinitely many generators) cannot be NIP. 





This philosophy can be pushed even further, for instance by considering one of the most venerable open question in the study of $n$-Engel groups. Recall that a group is $n$-Engel if it satisfies the identity 
\[
[\ldots [x,\underbrace{y],y],\ldots,y}_{n\text{ times}}] = 1.
\]
It is still open today whether $n$-Engel groups are locally nilpotent \cite{Traustason2011EngelGroups}. Counterexamples to this question -if they exist- seem unlikely to be tame, and it is therefore quite safe to conjecture that they are not.

\begin{conjecture}
An NIP $n$-Engel group is locally nilpotent.
\end{conjecture}

Note that free $n$-Engel groups are not NIP in general. Together with $(\dagger)$ above, the veracity of the conjecture would yield that $n$-Engel NIP groups are nilpotent. This is known for $n$-Engel stable groups, by the work of Derakhshan and Wagner \cite{DerakhshanWagner1997}. This is also known for some cases of $\omega$-categorical NIP groups by Krupinski \cite{Krupinski2012}, Dobrowolski and Wagner \cite{DobrowolskiWagner2020}. As for groups with a simple theory, Hempel \cite{HEMPEL2020169} proves $(\dagger)$, leaving open the conjecture for groups with simple theories.

\subsection{Results} The first result of this paper is Theorem \ref{thm:NIPringsNilpotent}: NIP rings of bounded nilexponent are nilpotent. If the proof is short in the commutative case, the general case demands a little more work and is essentially using the heritability for definable quotients of the Baldwin-Saxl condition on annihilators. It relies on a careful analysis of \textit{element of second order} (in the sense of Zelmanov \cite{ZelmanovNilrings1992}), that is, elements that generates a null ideal. Their asymptotic product is shown to vanish uniformly by a combinatorial argument à la Ramsey (Corollary \ref{cor:SecondOrderNilpInBoundedNilexp}). A crucial lemma of Cherlin \cite{cherlinnilringII} yields that second order elements are abundant in nilrings of bounded nilexponent, which allows to conclude. Pushing the analysis a bit further, our results extends to higher-arity tameness conditions: we prove that NIP$_2$ rings of bounded nilexponent are nilpotent. The authors believe that those results can be extended to the case of NTP$_2$ and NIP$_n$ for all $n\in \N$. This will be investigated in future work.

We then turn to NIP Lie algebras. Following the strategy of transfer between categories mentioned above, we prove that some Baldwin-Saxl chain condition transfers from an NIP Lie algebra $L$ to its enveloping algebra $A(L)$, that is, the associative subalgebra of $\End(L)$ generated by the adjoint maps $x\mapsto [x,a]$ for $a\in L$. Note that it seems unlikely that NIP transfers directly from $L$ to $A(L)$. This allows us to prove that some NIP \textit{sandwich algebras} (in the sense of Kostrikin \cite{kostrikinAroundBurnside}) are nilpotent, see Theorem \ref{thm:kstrongtoastieomegaNIP}. This result allows us to deduce a particular case of the Wilson conjecture for $\omega$-categorical Lie algebras: NIP $\omega$-categorical $n$-Engel Lie algebras over a field of characteristic $p>n$ are nilpotent.

\subsection{The case of NSOP nilrings} As an illustration of the methods and phenomenon that appear in this paper, we quickly explain how $(\dagger)$ is obtained for another important tameness property: NSOP.
This already appears in the literature, see for instance the work of Berline \cite{Berline1979Anneaux}. Recall that a theory is NSOP if there is no formula $\phi(x,y)$ which defines, in some model of the theory, an strict order with infinite chain. Consider the formula expressing the left dividing relations $x\mid y$ that is, the formula $\exists z \ y = xz$. Then in a nilring $A$, the relation $x\mid y$ is strict on the set $A\setminus \set{0}$, and by associativity, it is also transitive. By compactness, there exists a number $s$ such that $A\setminus \set{0}$ has no $\mid$-chain of length greater than $s$. Picking $a_1,\ldots, a_{s+1}\in A\setminus \set{0}$, one of the element in the chain 
\[a_1\mid a_1a_2\mid \ldots \mid a_1\cdots a_{s+1}\]
must be zero, so $x_1\cdots x_{s+1} = 0$ is an identity of $A$. Quite surprisingly, here the nilexponent need not be uniformly bounded. 

\subsection*{Acknowledgement} This work was initiated at the Université des Sciences et Techniques de Masuku (Gabon) during a visit by the first-named author to the last-named author. This stay was funded by the European Union’s Horizon Europe Marie Skłodowska-Curie Actions (MSCA) — Staff Exchanges program Mathematics in Gabon for Africa (MaGA). The authors express their warm thanks to Ariane Mézard and Cécile Armana for coordinating this wonderful initiative.

\subsection*{AI Disclosure} No AI was used for establishing the results of this paper, nor for the writing and proofreading.

\subsection*{Notations and conventions.} A \textit{nilring} is a ring $R$ such that for each $a\in R$ there is $n\in \N$ such that $a^n = 0$. If the $n$ is uniform for every element or $R$, the ring is said to be of \textit{bounded nilexponent}, and the nilexponent is the smallest such $n$. The ring $R$ is \textit{nilpotent} if there is $k\in \N$ such that $a_1\cdots a_k = 0$ for all $a_i\in R$. A nilring has no unit element. An associative algebra is a ring where the additive group is a vector space over some field. A \textit{semigroup} is a set equipped with an associative composition law, and with a zero element, typically the multiplicative structure of a ring is a semigroup. Model-theoretic notations are standard.

\section{NIP groups and chain conditions}
\subsection{NIP theories} Recall that a theory $T$ is called \textit{NIP} if there is no model $M\models T$, formula $\varphi$ and $( a_i)_{i<\omega}\in M^\omega, ( b_J)_{J\subseteq \omega}\in M^{\mathcal{P}(\omega)}$ such that for all $i<\omega$ and $J\subseteq \omega$ we have $M\models \varphi(a_i,b_J)\iff i\in J$. 
We say a structure $M$ is NIP if its theory is. Classical examples of NIP structures includes stable structures, o-minimal structures, ordered abelian groups and various classes of Henselian valued fields \cite{Delon}.


\begin{definition}
For a given formula $\varphi(x, y)$, let us call the \textit{Vapnik-Chervonenkis dimension} of $\varphi$ the integer $\operatorname{vc}(\varphi)$ defined by: 
\[\operatorname{vc}(\varphi( {x}; {y}))=\sup\{n\in \N\mid  \exists ( {a}_i)_{i<n}, \exists ( {b}_J)_{J\subseteq \llbracket 0,n-1\rrbracket}, \models \varphi(a_i,b_J)\iff i\in J\}\in \N\cup \{\infty\}\]
\end{definition}
By compactness, a theory $T$ is NIP if and only if every formula is of finite VC-dimension. We also recall the dual VC-dimension defined by $\operatorname{vc}^*(\varphi)=\operatorname{vc}(\varphi^\text{opp})$. 
Classically we have $\operatorname{vc}(\varphi)<\infty\iff \operatorname{vc}^*(\varphi)<\infty$ and more precisely $\operatorname{vc}^*(\varphi)<2^{\operatorname{vc}(\varphi)+1}$ (and conversely). This notion will be useful as it will appear when introducing the classical chain condition on NIP groups.

In this paper we will also consider the following $n$-ary version of the NIP condition.

\begin{definition}
A theory $T$ is called \textit{$n$-NIP} if there is no model $M$, formula $\varphi( {x};  y_0,\dots,  y_{n-1})$ and $(  a_I)_{I\subseteq \omega^n}, (  b_i^l)_{i<\omega, \\ l<n}$ such that $M\models \varphi(  a_I,   b_{i_0}^0,\dots,  b_{i_{n-1}}^{n-1})\iff (i_0,\dots,i_{n-1})\in I$.  
\end{definition}

Note that we have the following implication:
\[
\text{NIP}=1\text{-NIP}\implies 2\text{-NIP} \implies \cdots \implies k\text{-NIP} \implies (k+1)\text{-NIP} \implies \cdots
\]
Those implications are strict: the generic $c$-nilpotent $p$-group for $p>c$ is $c$-NIP but not $(c-1)$-NIP: see \cite{dElbeeMullerRamseySinioraModelTheoreticNilp}. 
\subsection{Baldwin-Saxl condition} An important tool in the study of NIP groups is the Baldwin-Saxl chain condition. In this section we will present the chain condition, compare it with the icc (defined below), and also present the equivalent condition for $n$-NIP groups.
\begin{fact}
    Let $G$ a NIP group. For every $\varphi$, there is a $k_\varphi$ such every finite intersection of $\varphi$-definable subgroup is the intersection of at most $k_\varphi$ of them. More precisely, for every $a_1,\dots ,a_n$ such that $\varphi(G,a_i)\leq G$, there is $I\subseteq \llbracket 1,n\rrbracket$ with $\lvert I\rvert \leq k_\varphi$ such that \[
    \bigcap_{1\leq i\leq n}\varphi(G,a_i)=\bigcap_{i\in I}\varphi(G,a_i).
    \]
\end{fact}
\begin{remark}
    The result is clearly false if we are not dealing with groups, even in a strongly minimal
 context: let $X$ be an infinite set, and $\varphi(x,y)\equiv x\neq y$. Then $\bigcap_{a\in A} \varphi(X,a)=\bigcap_{b\in B} \varphi(X,b)\iff A=B$.
\end{remark}
While we saw that the intersection of a finite family of uniformly definable subgroups is the intersection of a bounded number of subgroups, there is no reason to believe the intersection of $\varphi$-definable subgroups is the intersection of a bounded number of them. This condition is called the \textit{intersection chain condition} (icc). Let rapidly give an example of NIP group not verifying icc:
\begin{example}
    Let $(K,v)$ be a NIP valued field of maximal ideal $\m$. Suppose its value group is densely ordered. Then if we let $I_\varepsilon=\{x\in K\mid v(x)>\varepsilon\}$, we have $\bigcap_{\varepsilon>0}I_{\varepsilon}=\m$ yet a finite intersection of $I_{\varepsilon_i}$ with $\varepsilon_i>0$ is $I_{\min \varepsilon_i} \neq \m$.
\end{example}
The \textit{uniform chain condition} (ucc) is when chains of $\varphi$-definable of subgroups are of bounded length. It is readily seen that groups with ucc verifying the Baldwin-Saxl condition verify icc. As NSOP groups and rosy groups have the ucc (see \cite{EalyCliftonKrupinskiPillaySupperosyNIPGroups}), we deduce that stable groups and NIP rosy groups have the icc. 

Let us also recall the usual chain condition for $n$-NIP groups, first proved by Hempel in \cite{hempel2016n}.
\begin{fact}
    Let $G$ be a $n$-NIP group. Let $H(  x_0,\dots,  x_{n-1})$ be definable family of subgroup. There is an integer $m_H$ such that for every integer $d\geq m_H$ and every $(a_{i,j})_{0\leq i \leq n-1, \\ 1\leq j\leq d}$, there is a $\eta\in d^n$ such that \[\bigcap_{\nu \in d^n} H(a_{0,\nu(0)},\dots, a_{n-1,\nu(n-1)})=\bigcap_{\nu \neq \eta} H(a_{0,\nu(0)},\dots, a_{n-1,\nu(n-1)}) \]
\end{fact}

\subsection{Breadth of a lattice}
Most of usual chain conditions can be seen as an order-theoretic condition on certain families of subgroups: a ring is Artinian if its poset of ideals is well-founded, is Noetherian if its poset of ideals is co-well-founded, a group is ucc if its families of uniformly definable subgroups are of finite length, etc. The Baldwin-Saxl condition does not fail this rule: this condition corresponds to some finiteness condition on families of uniformly definable subgroups, the condition of being of \textit{finite breadth}. 

\begin{definition}
A $\wedge$-semilattice $X$ is said to be \textit{of finite breadth} if there is an integer $r$ such that for every $x_1,\dots,x_r\in X$, we have $x_1\wedge \cdots \wedge x_r = x_1\wedge \cdots \wedge \hat{x_i}\wedge \cdots \wedge x_r$. Equivalently the meet of a finite number of elements is the meet of at most $r$ of them. The smallest such $r$ is called the \textit{breadth} of $X$ and is noted $\operatorname{br}(X)$. 
\end{definition}

With this notion, the Baldwin-Saxl condition is nothing but asking the $\wedge$-lattice generated by a family of uniformly definable subgroups to be of finite breadth. Moreover, the breadth of the $\wedge$-semilattice $\{\bigcap_i \varphi(G, {a}_i)\}$ is less than the VC$^*$-dimension of $\varphi$ (and is equal if $G$ is the home sort of the structure). 
Note that the definition of breadth can be simplified when talking about lattice: a lattice is of breadth $\leq n$ if and only if it does not contain a copy of $\mathcal{P}(n)$.

This notion of breadth appears in algebra not only in the Baldwin-Saxl condition but in a lot of other places:

\begin{itemize}
    \item In \cite{ShevrinOvsyannikov}, Shevrin showed that a locally finite group is \textit{of finite breadth} (i.e. its lattice of subgroups is), if and only if $G$ is a Chernikov group.
    \item In \cite{HaleviPalacinDpRankAbelianGroups}, Halevi and Palacín showed that an abelian group $A$ is of finite dp-rank if and only if its semilattice of commensurability classes of the p.p.-definable subgroups is of finite breadth, and when they are both finite those agree. Moreover they generalize this result to the case of one-based groups.
    \item In \cite{JohnsonDpfiniteNoetherianNIPintegraldomains}, Johnson studies (commutative) rings (with unit) whose lattice of ideals of $R$ is of finite breadth. Those rings are called the $W$-rings, and the associated breadth is the \textit{weight} of the ring. Examples of such rings includes dp-finite (commutative unital) rings and Noetherian NIP (commutative unital) rings.  Assuming the henselianity conjecture, Johnson shows $W$-rings are direct products of finitely many henselian local rings.
    \item In \cite{MurzinovaRingsFiniteBreadth}, Murzinova studies rings whose lattice of subrings of $R$ is of finite breadth. Those rings, usually called \textit{rings of finite breadth}, were classified by Murzinova.
\end{itemize}
We clearly have the following implication: 
\[
\text{of finite breadth} \implies W\text{-ring} \implies \text{Baldwin-Saxl Condition (on ideals)}
\]
\begin{remark}
    We can see the notion of $W$-ring as an uniform Baldwin-Saxl condition on ideals. To see that let $\mathcal{C}$ be a $\wedge$-semilattice of definable subgroups of a group $G$. Then $\operatorname{br}(\mathcal{C})\leq n$ if and only if every family of uniformly definable subgroups in $\mathcal{C}$ verify the Baldwin-Saxl condition of breadth $\leq n$. Thus in particular, by taking $\operatorname{Sub}_R^\text{def}(R)$ to be the definable ideals of $R$, we have $\operatorname{br}(\operatorname{Sub}_R^\text{def}(R))\leq n$ if and only if every family of uniformly definable ideals of $R$ verify the Baldwin-Saxl condition of breadth $\leq n$. But by \cite[Lemma 5.4]{JohnsonDpfiniteNoetherianNIPintegraldomains}, $\operatorname{br}(\operatorname{Sub}_R^\text{def}(R))=\operatorname{br}(\operatorname{Sub}_R(R))$.
\end{remark}

\section{Nilrings}

\subsection{Non-nilpotent nil NIP rings}\label{sec:locniltame}
We first give a simple example of a non-nilpotent commutative NIP nilring, showing that we cannot hope for direct translation of the NSOP result of the introduction in the NIP case. We will in fact show that there are dp-minimal examples of non-nilpotent nilrings. The main result of this paper, namely that NIP rings of finite nilexponent are nilpotent, indicates that nilrings of unbounded nilexponents are the only possible examples.

Consider a valued field $(K,v)$, with value group $\Gamma$, valuation ring $\cO$ and maximal ideal $\m$. Fix $\epsilon>0$ in $\Gamma$ and set $I_\epsilon = \set{a\in \cO, v(a)>\epsilon}$. Then, for any $a\in \m$ and $n\in \N$, we have 
\[(a+I_\epsilon)^n = 0 \ \text{ (in $\m/I_\epsilon$)} \iff nv(a)>\epsilon.\]
We have the following conclusion concerning the rings $\m/I_\epsilon$:
\[\m/I_\epsilon \text{ is nil for any $\epsilon\in \Gamma$}\iff \Gamma \text{ is archimedean}.\]
Assuming that the valued group $\Gamma$ is densely ordered, for each $n\in \N$ there is $a\in \m$ such that $a^n<\epsilon$, hence $\m/I_\epsilon$ has elements of arbitrary large nilexponent, in particular it is not nilpotent.
In any valued field, the maximal ideal is definable and so is the ideal $I_\epsilon$ using $\epsilon $ as parameters. Therefore we conclude.
\begin{example}\label{ex:counterexamplenilnotnilp}
    Let $(K,v)$ be a valued field with archimedean, densely ordered value group $\Gamma$ and let $I_\epsilon$ be the ideal of elements of value greater than $\epsilon$, for some positive $\epsilon\in \Gamma$. Then the ring $\m/I_\epsilon$ is an interpretable nilring which is not nilpotent. In particular, if $K$ is dp-minimal/NIP/NTP$_2$, then so is $\m/I_\epsilon$. This includes the cases where $K$ is the Hahn series $\C((\Q))$ or $\F_p^\alg((\Q))$, see e.g. \cite[Theorem 9.8.1]{Joh16}.
\end{example}




\begin{remark}
        From Example \ref{ex:counterexamplenilnotnilp} we deduce counterexample to a few related questions on NIP rings: since the Jacobson radical of $R=\cO/I_\epsilon$ is $\m/I_\epsilon$, we have examples of dp-minimal commutative unital rings whose Jacobson radical is of infinite nilexponent. By taking an $\omega$-saturated extension of $R$, we conclude that it is possible for a dp-minimal commutative unital ring to have a non nil Jacobson radical. However, by Theorem \ref{thm:NIPringsNilpotent}, for an $\omega$-saturated NIP ring $R$, the Jacobson radical $J(R)$ is nil if and only if $J(R)$ is nilpotent.
\end{remark}

\begin{remark}
    The authors believe that a noncommutative counterexample could also be obtained by considering dp-finite division algebras \cite{delbee2021cyclicnoncyclicdivisionalgebras}. From a noncommutative dp-finite non-nilpotent nilring, a dp-finite locally nilpotent nonnilpotent group shall be obtained by considering the adjoint group.
\end{remark}

\subsection{The commutative case} In this subsection, we will deal with the commutative case of our main result. The proof given here will serve as a guide for the proof in the general case. In fact, we actually show that the Baldwin-Saxl condition restricted to  annihilators suffices for treating the commutative case. We also obtain a bound on the nilpotency class of our rings depending only on the vc$^*$-dimension of $xy=0$. 

\begin{definition}
A ring $R$ is satisfies the \textit{Baldwin-Saxl condition on annihilators} if there exists $k\in \N$ such that for all $a_1,\dots,a_m\in R$ with $m\geq k$, there is $i_1<\cdots<i_{k}$ such that $\operatorname{RAnn}(a_1,\dots,a_m)=\operatorname{RAnn}(a_{i_1},\dots,a_{i_{k}})$. Such a ring will be called a \textit{BS$_0$ ring}, and the \textit{breadth} of such ring is the smallest such $k$.
\end{definition}   

\begin{remark}
    We here deal with the \textit{right} BS$_0$ condition and in this subsection we will only be dealing with the commutative case anyways. However, in the non-commutative case, because the breadth associated with BS$_0$ is nothing but the vc$^*$-dimension of $xy=0$, we see that a ring is right BS$_0$ if and only if it is left BS$_0$, and if the ring is right BS$_0$ of breadth $\leq k$ then the ring is left BS$_0$ of breadth $< 2^{k+1}$ (and reciprocally).
\end{remark}
Let us start with an elementary fact on BS$_0$ rings of finite nilexponent.  
\begin{lemma}\label{lm:basicBS0}
    Let $R$ be a commutative ring of finite nilexponent $n$. Suppose $R$ is BS$_0$ of breadth $\leq k$. 
    Then for every $ a_1, \dots,  a_{k+1}\in R$, we have $ a_1^{n-1}\cdots a_{k+1}^{n-1}=0$. 
\end{lemma}
\begin{proof}
    Let $R$ be such a ring et let  $a_1,\dots,a_{k+1}\in R$.  
    
    $\triangleright$ By BS$_0$, there is $i\in \llbracket 1,k+1\rrbracket$ such that $\operatorname{Ann}(a_1^{n-1},\dots,a_{k+1}^{n-1}) = \operatorname{Ann}(a_1^{n-1},\dots, , a_{i-1}^{n-1}, a_{i+1}^{n-1},\dots, a_{k+1}^{n-1})$. \\
    But $a_1\cdots a_{i-1}a_{i+1}\cdots a_{ k+1}\in\operatorname{Ann}(a_j^{n-1}\mid  j\neq i) = \operatorname{Ann}(a_1^{n-1},\dots,a_{ k+1}^{n-1})$, thus $a_1\cdots a_{i-1}a_i^{n-1} a_{i+1}\cdots a_{ k+1}=0$.  
    Thus finally $a_1^{n-1}\cdots a_{ k+1}^{n-1}=0$. 
\end{proof}
The lemma above, and an induction permits us to get the following.
\begin{proposition}
    Commutative NIP rings of finite nilexponent are nilpotent.
\end{proposition}
\begin{proof}
    Let $R$ be a commutative NIP ring of finite nilexponent $n$. By Lemma \ref{lm:basicBS0}, there exists $k$ such that for every $x_1,\dots,x_k\in R$, we have $x_1^{n-1}\cdots x_k^{n-1}=0$. \\

    $\blacktriangleright$ Assume by induction that NIP rings of nilexponent $\leq n$ verifying $\forall \overline{x},  x_1^{n-1}\cdots x_{k-1}^{n-1}=0$ are nilpotent. \\
    
    $\triangleright$ Without loss of generality, let us assume that $n\geq 3$ and $k\geq 2$. \\

    $\blacktriangleright$ Let $I=\{a\in R\mid \forall b\in R, \ ab^{n-1}=0\}$. 
    It's readily checked that $I$ is a definable ideal of $R$, and that:
    \begin{itemize}
        \item $\forall a,b\in I, a^{n-1}b^{n-1}=0$, and
        \item  $\forall c_1,\dots,c_{k-1}\in R/I, \ c_1^{n-1}\cdots c_{k-1}^{n-1}=0$.
    \end{itemize} 
    Thus by induction hypothesis, $R/I$ is nilpotent, and if $k>2$ then $I$ is also nilpotent. \\

    $\triangleright$ Let us show that $I$ is nilpotent even when $k=2$.  \\
    Let $J=\{a\in I\mid \forall b\in I,\ ab=0\}$. Then $I/J$ is NIP of nilexposant $<n$ thus by induction, $I/J$ is nilpotent, and thus $I$ is nilpotent. \\

    $\triangleright$ Thus $I$ and $R/I$ are nilpotent, thus $R$ is nilpotent.
\end{proof}

This last result was shown in two parts: firstly, we showed that for every $\operatorname{BS}_0$ nilring of nilexponent $n$, there is a $k$ such that $\forall x_1,\dots,x_k, \ x_1^{n-1}\cdots x_k^{n-1} = 0$, and finally we do an induction using the crucial fact that definable sections of a NIP ring are BS$_0$. It is natural to wonder if the seemingly crucial assumption on definable sections is necessary, in other words, if there are non-nilpotent commutative BS$_0$ nilring of finite nilexponent.

It is not too hard to convice oneself that it should not be the case: it is clear that BS$_0$ rings of nilexponent $2$ are nilpotent, and by induction on the nilexponent, since $\{x^2\mid x\in R\}$ is of smaller nilexponent, we should be able to conclude. However the set $\{x^2\mid x\in R\}$ is in general not a subring of $R$, making the previous idea fail as long as we only consider the induction on nilrings. To overcome this issue, we shall make use of the notion of \textit{semigroup}: more precisely, we will consider the natural notion of BS$_0$ over a semigroup, hence all our semigroups have a zero element denoted $0$. 
Note that $\{x^2\mid x\in R\}$ is a sub-semigroup of $(R,\cdot)$. 

\begin{definition}
A semigroup $M$ is called \textit{BS$_0$} if there exists $k\in \N$ such that for all $a_1,\dots,a_m\in M$ with $m\geq k$, there is $i_1<\cdots<i_{k}$ such that $\operatorname{RAnn}(a_1,\dots,a_m)=\operatorname{RAnn}(a_{i_1},\dots,a_{i_{k}})$. The \textit{breadth} of $M$ is the smallest such $k$.
\end{definition}   

\begin{remark}
    While a sub-semigroup of the multiplicative semigroup of an NIP ring do verify BS$_0$, there exists non BS$_0$ NIP semigroups. Let $X$ be an infinite set and let $\cdot$ defined on $X\sqcup\{0\}$ by $x\cdot y= 0$ if $x\neq y$ or $x=0$ or $y=0$, and $x\cdot x=x$ otherwise. 
    Then $(X\sqcup\{0\},\cdot)$ is a strongly minimal (thus NIP) non BS$_0$ semigroup. 
\end{remark}

 The key fact we use for BS$_0$ commutative nilrings is easily seen instead as a result on commutative BS$_0$ semigroups.
\begin{lemma}
Let $M$ be a commutative semigroup of nilexponent $2$. If $M$ verifies BS$_0$ (of breadth $k$), then $M$ is nilpotent (of class $\leq k+1$).
\end{lemma}
\begin{proof}
For any $a_1,\ldots,a_{k+1}\in A$, there exists $1\leq i\leq k+1$ such that $\operatorname{Ann}(a_1,\ldots,a_{k+1}) = \operatorname{Ann}(a_1,\ldots,\hat{a_i},\ldots, a_{k+1})$. In particular, 
\[a_1\cdots \hat{a_i}\cdots a_{k+1}\in\operatorname{Ann}(a_1,\ldots,\hat{a_i},\ldots, a_{k+1})\seq \operatorname{Ann}(a_1,\ldots,a_{k+1})\seq \operatorname{Ann}(a_i)\]
hence $a_1\cdots a_{k+1} = 0$. 
\end{proof}

From this lemma and the idea above, we easily deduce the following result: 
\begin{proposition}
Let $M$ be a commutative semigroup of finite nilexponent $n$. If M is BS$_0$, then $M$ is nilpotent.
\end{proposition}
\begin{proof}
    We proceed by induction on the nilexponent. \\
    Assume BS$_0$ semigroup of nilexponent $\leq m$ are nilpotent and let $M$ be a BS$_0$ semigroup of nilexponent $\leq 2m$. 

    Then by induction hypothesis, $B=\{x^2\mid x\in M\}$ is nilpotent and so let $k$ be its class of nilpotency. \\
    Then $(x_1\cdots x_k)^2=x_1^2\cdots x_k^2 = 0$, thus $P_k=\{x_1\cdots x_k\mid \overline{x}\in M^k\}$ is a $\operatorname{BS}_0$ semigroup of nilexponent $\leq 2$. \\
    Thus by lemma $P_k$ is nilpotent, and thus $M$ is nilpotent. 
\end{proof}
In particular, commutative BS$_0$ rings of finite nilexponent are nilpotent.

\begin{remark}
    A variation of above shows that if $M$ is BS$_0$ of breadth $r$, then $M$ is nilpotent of class $\leq (r+1)^{\lceil \log_2(n)\rceil}$. 
\end{remark}
\subsection{Second-order elements} The proof we gave in the commutative case was crucially using the identity $a^{n-1}xa = 0$ for every $a,x$ in a commutative ring of nilexponent $n$. Although there is no reason to believe that this last identity holds for any non-commutative ring of nilexponent $n$, we can always consider the elements $b$ such that $bRb=0$ and replace in our proof the $(n-1)$th powers by those elements. That way, the part taken by the  $(n-1)$th powers in the commutative case will be played in the non-commutative case by elements called the \textit{second order elements}, see \cite{ZelmanovNilrings1992}.

\begin{definition}
An element $a$ of a ring $R$ will be called a \textit{second order element} of $R$ if $a^2=0$ and $axa=0$ for all $x\in R$.
\end{definition}

Those elements are classically used to solve the Kurosh-Levitzki problem in the finite nilexponent case: if $a\in A$ is of second order and $I$ is a locally nilpotent ideal, then the ideal $\langle I,a\rangle$ is locally nilpotent. Thus the locally nilpotent radical of the ring contains all the second order elements. When combined with the following lemma, it shows that rings of finite nilexponent are indeed locally nilpotent: 
\begin{fact}\cite[Lemma 6]{cherlinnilringII}
    Let $R$ be a ring of finite nilexponent $n$, then $R$ satisfies the identity $x^{n-1}yx^{n-1} = 0$. In other words, $a^{n-1}$ is of second order in $R$ for all $a\in R$.
\end{fact}


To our knowledge, the previous lemma was first shown by Cherlin in \cite[Lemma 6]{cherlinnilringII}. It was already known to Kaplansky in the case of nilalgebras over fields with $\geq n$ elements (see \cite{Kaplansky1946}, in the proof of the lemma 3). We will also need the following classical fact about nilrings of bounded nilexponent. See for instance \cite[(1.1), p. 14]{kostrikinAroundBurnside} or \cite[Proof of Lemma 1.2.7]{VaughanLee1993}.

\begin{fact}\label{fact:standardID}
    Let $R$ be a ring of finite nilexponent $\leq n$. Then 
    $\forall a_1,\dots,a_n\in R, \ \sum_{\sigma\in \mathfrak{S}_n}a_{\sigma(1)}\cdots a_{\sigma(n)}=0$.    
\end{fact}
\begin{proof}
This follows from the following equality which holds in any ring:
\[\sum_{\sigma\in \fS_n}y_{\sigma(1)}\cdots y_{\sigma(n)} = \sum_{\emptyset\neq S\seq \set{1,\ldots, n}} (-1)^{n-\abs{S}} (\sum_{i\in S} y_i)^{n}\]
To convince oneself, see the case $n = 3$:
    \begin{align*}
        (y_1+y_2+y_3)^3 = \sum_{\sigma\in\fS_3} y_{\sigma(1)} y_{\sigma(2)} y_{\sigma(3)} + (y_1+y_2)^3+(y_1+y_3)^3+(y_2+y_3)^3 - y_1^3 - y_2^3 - y_3^3.
    \end{align*}
\end{proof}

\begin{theorem}\label{thm:secondorderelementsinBS}
    Let $M$ be a BS$_0$ semigroup and let $(a_i)_{i<\omega}$ be a sequence of second order elements. Then there is a sequence $(j_i)_{i<\omega}$ such that $j_0=1$ and by letting $b_l=a_{j_{l-1}}\cdots a_{j_l}$, we have $b_ib_j=0$ for all $i\geq j$.
\end{theorem}

\begin{proof}
Let $M$ and $(a_i)_{i<\omega}$ be as in the hypotheses and let $k$ be the breadth of $M$. Fix $m>k$. Then, by BS$_0$, for any $a_{i_1},\dots,a_{i_m}$ second order elements there exists $1\leq j\leq m$ such that $\operatorname{RAnn}(a_{i_1},\dots,a_{i_{m}})=\operatorname{RAnn}(a_{i_1},\ldots, \hat{a_{i_j}}, \ldots, a_{i_{m}})$
and hence \[a_{i_j}a_{i_{j+1}}\cdots a_{i_m}a_{i_1}\ldots a_{i_{j-1}} = 0.\]

Denoting $[\N]^m$ the set of $m$-elements subsets of $\N$ and $F:[\N]^{m} \to \set{1,\ldots, m}$ be a function such that for every $i_1<\cdots <i_{m}$, we have $\operatorname{RAnn}(a_{i_1},\dots,a_{i_{m}})=\operatorname{RAnn}(a_{i_j}\mid j\neq F(i_1,\dots,i_m))$.
    
    $\blacktriangleright$ By Ramsey theorem, there is an infinite subset $H$ of $\N$ such that $F$ is constant on $[H]^{m}$. Let $c$ be the value of $F$ on $[H]^{m}$. Then, for every $ i_1,\dots, i_{m}\in H$ with $i_1<\cdots <i_{m}$, we have
    \[
    \operatorname{RAnn}(a_{i_1},\ldots,a_{i_m})=\operatorname{RAnn}(a_{i_j}\mid j\neq c).
    \]
    In particular, if $u$ is any monomial involving the $(a_{i_j})_{j\neq c}$'s, then $a_{i_c}\cdot u=0$.  \\

    $\blacktriangleright$ Let $(h_l)_{l\geq 1}$ be an increasing enumeration of $H$. Let $j_l = h_{ml}$ for $l>0$ and $j_0 = 1$. Let $(b_i)_{1\leq i<\omega}$ be defined as $b_l = a_{j_{l-1}}a_{j_{l-1}+1}\cdots a_{j_l}$.
    Let $l>s\geq 1$, we show that $b_lb_{s}=0$. By hypothesis, we know that 
    \[
    F(h_{m s},\cdots, h_{ms+c-1},h_{ml},\cdots, h_{ml+(m-c)})=h_{ml}
    \]
    thus $a_{h_{ml}}\cdot u =0$ if $u$ is a product of terms containing the $a_{h_t}$ for every $t$ in the set 
    \[\set{ms,\dots, ms+c-1,ml+1,\cdots, ml+(m-c)}.\]
    As we can write $b_lb_s=a_{h_{ml}}\cdot u$ with $u$ containing those $a_{h_t}$, we conclude $b_lb_{s}=0$.
\end{proof}

We immediately deduce the following.

\begin{corollary}\label{cor:BS0ringsproductsecondorder}
    Let $R$ be a BS$_0$ ring verifying an identity of the form $x_1\cdots x_r+\sum_{\sigma\neq \operatorname{Id}} \lambda_{\sigma}x_{\sigma(1)}\cdots x_{\sigma(r)}=0$. \\
    Then there is a $k$ such that for every second order elements $a_1,\dots,a_k$, we have $a_1\cdots a_k=0$. 
\end{corollary}
\begin{proof}
    As being BS$_0$ is a first-order condition on the ring, and the set of second order element is definable, we may assume that the ring $R$ is $\omega$-saturated. Then it suffices to show that for any sequence $a = (a_i)_{i<\omega}$ of second order elements of $R$ there is some number $k = k(a)$ such that $a_1\cdots a_k = 0$. Let $(j_i)_{i<\omega}$ be as in Theorem \ref{thm:secondorderelementsinBS} and let $k = j_r$, where $r$ is as in the hypotheses. Then 
    \[
    a_1\cdots a_{k} = b_1\cdots b_r = \sum_{\sigma\neq \operatorname{Id}} \lambda_{\sigma}b_{\sigma(1)}\cdots b_{\sigma(r)}=0
    \]
    which finishes the proof.
\end{proof}
\begin{remark}
    It is possible to give a purely algebraic proof of the preceding corollary: in the theorem \ref{thm:secondorderelementsinBS}, if we replace the use of the infinite Ramsey theorem by the finite Ramsey theorem, we can state a finite version of theorem \ref{thm:secondorderelementsinBS}. While it can be said to be nothing but as a straightforward application of the compactness theorem, it allows both a pure algebraic proof and to find an explicit bound $k$ depending on a Ramsey number of a hypergraph, depending on $r$ and the breadth of BS$_0$. 
\end{remark}

Note that every polynomial identity implies a multilinear identity, therefore we obtain the following refinement for PI algebras.
\begin{corollary}\label{cor:BS0PIalgebrassecondorder}
    Let $A$ be a BS$_0$ associative algebra satisfying a polynomial identity.
    Then there is a $k$ such that for every second order elements $a_1,\dots,a_k$, we have $a_1\cdots a_k=0$. 
\end{corollary}

Together with Fact \ref{fact:standardID} we obtain the following.

\begin{corollary}\label{cor:SecondOrderNilpInBoundedNilexp}
    Let $R$ be a BS$_0$ ring of bounded nilexponent, then there exist $k$ such that for every second order elements $a_1,\dots,a_k$, we have $a_1\cdots a_k=0$. 
\end{corollary}


\subsection{NIP nilrings}
In the commutative case, once we knew that $\forall a_1,\dots, \forall a_k \in , \ a_1^{n-1}\cdots a_k^{n-1}=0$, we could conclude by induction on $n$ and on $k$: 

Indeed, by letting $ I=\{a\in R\mid \forall b\in R, \ ab^{n-1}=0\}$, we got a definable ideal of $R$ verifying that $\forall a,b\in  I, a^{n-1}b^{n-1}=0$, and that $\forall c_1,\dots,c_{k-1}\in R/ I, \ c_1^{n-1}\cdots c_{k-1}^{n-1}=0$, which allowed us to reduce to the case where $k=2$ which was treated similarly. 
The problem with that idea is that, despite that we know that there's a $k$ such that the product of $k$ elements of second order is always zero (and thus $\forall x_1,\dots,x_k\in R, \ x_1^{n-1}\cdots x_k^{n-1}=0$), the set $ I$ defined earlier is a priori only a left ideal and not a two-sided one.  

To adapt the ideal $ I$ in a way that still work, we want to consider $ I=\operatorname{RAnn}(X)$ with $X$ a set of second order elements, stable by multiplication on the left (and then we will replace the induction hypothesis by $\forall a_1,\dots ,a_k\in X(R), \ a_1\cdots a_k=0$.) 
Taking the set $X$ to be the whole set of second order elements do not seems to work: there is a priori no reason to believe that $a_1\cdots a_{k-1} = 0$ for every $a_1,\dots,a_{k-1}\in R/ I$ of second order: new second order elements could appear in the quotient (an example would be a ring where $a^{n-1}b^{n-1}=0$: the quotient thus would be of nilexponent $\leq n-1$, and thus the $x^{n-2}$ are new second order elements not necessarily verifying $x=0$). 
However an argument closer to the commutative one can work: if we take $X$ to be the set $ \{bx^{n-1}\mid b\in R\cup\{1\}, \ x\in R\}$, $ I=\operatorname{RAnn}(X)$ and induce on the $k$ such that $\forall a_1,\dots, a_k\in R\cup\{1\}, \ \forall x_1,\dots, x_k\in R, \ a_1x_1^{n-1}\cdots a_k x_k^{n-1}=0$, we will be able to conclude. 

Using this idea we are able to prove our result.

\begin{theorem}\label{thm:NIPringsNilpotent}
    NIP rings of finite nilexponent are nilpotent.
\end{theorem}
\begin{proof}
    We proceed by induction on the nilexponent $n$, then on the integer $k$ such that \[\forall a_1,\dots, a_k\in R\cup\{1\}, \ \forall x_1,\dots, x_k\in R, \ a_1x_1^{n-1}a_2x_2^{n-1}\cdots a_k x_k^{n-1}=0\]
    Such a $k$ always exists by Corollary \ref{cor:BS0ringsproductsecondorder} because the $a_ix_i^{n-1}$'s are second order elements in $R$. Assume without loss of generality that $n\geq 2$ and $k\geq 2$.
    
    $\blacktriangleright$  Let $I = \{a\in R\mid \forall b\in R\cup\{1\}, \ \forall x\in R, \ abx^{n-1}=0\}$. \\
    Then $I$, being an annihilator, is a left ideal. It is even a two-sided ideal because it is the annihilator of the set $\{bx^{n-1}\mid b\in R\cup\{1\}, \ x\in R\}$, which is stable by multiplication on the left. \\

    $\triangleright$ Moreover $R/I$ verify that $\forall a_1,\dots, a_{k-1}\in R/I\cup\{1\}, \ \forall x_1,\dots, x_{k-1}\in R/I, \ a_1x_1^{n-1}\cdots a_{k-1} x_{k-1}^{n-1}=0$, \\
    and $I$ verify that $\forall a_1, a_2\in I\cup\{1\}, \ \forall x_1,x_2\in I, \ a_1x_1^{n-1} a_{2} x_{2}^{n-1}=0$ because $a_1x_1^{n-1}\in I$. \\

     $\triangleright$ Thus $A/I$ is nilpotent (by induction hypothesis), and if $k>2$ then we also know that $I$ is nilpotent. \\

     $\blacktriangleright$ We thus are left to show that $I$ is nilpotent when $k=2$.  \\
     Let in this case $J=\{a\in I\mid \forall b\in I, ab=0\} = \operatorname{RAnn}_{I}(I)$. That way, $J$ is a nilpotent two-sided ideal of $I$, and $I/J$ is of nilexponent $\leq n-1$. \\
     
    $\triangleright$ Thus we deduce from the induction hypothesis that $I/J$ is nilpotent, and thus so is $I$. We thus conclude that $R$ itself is nilpotent.
\end{proof}

As in the commutative case, we can ask whether BS$_0$ rings of finite nilexponent are nilpotent; the answer to this question is unknown to the authors.

\subsection{$n$-NIP rings} It is possible to generalize our results on nilpotency of NIP rings to the cases of $n$-NIP rings. We will more or less follow the proof of our previous theorem on NIP nilrings, but in an attempt of conciseness and of clarity, we will state without proof the following lemma which encapsulate the idea behind the induction made earlier: 
\begin{lemma}\label{lm:inductionStepNilp}
    Let $\mathcal{C}$ a class of rings stable under definable subrings and definable quotients. Suppose that for every ring $R\in \mathcal{C}$, the semigroup $S(R)$ of second order elements of $R$ is nilpotent. Then every ring of $\mathcal{C}$ of finite nilexponent is nilpotent.
\end{lemma}
\begin{remark}
    The proof of the theorem \ref{thm:NIPringsNilpotent} is nothing but the proof of this lemma combined with corollary \ref{cor:SecondOrderNilpInBoundedNilexp}.
\end{remark}
This lemma let us focus on the nilpotency of the semigroup $S(R)$ of second order elements, which is readily verified in the commutative case:
\begin{fact} \label{fact:commutativekNIPNilp}
    Commutative $n$-NIP rings of finite nilexponent are nilpotent.
\end{fact}
\begin{proof}
    Let $H(a_1,\dots,a_n)=\operatorname{Ann}(a_1\cdots a_n)$. By chain condition, we know there is an integer $m_H$ such that for every integer $d\geq m_H$ and every $(a_{i,j})_{0\leq i \leq n-1, \\ 1\leq j\leq d}$, there is a $\eta\in d^n$ such that \[\bigcap_{\nu \in d^n} H(a_{0,\nu(0)},\dots, a_{n-1,\nu(n-1)})=\bigcap_{\nu \neq \eta} H(a_{0,\nu(0)},\dots, a_{n-1,\nu(n-1)}) \]

    Let $nd$ elements $b_1,\dots,b_{nd}\in S(R)$ and let reindex them $\{a_{0,1},\dots, a_{0,d}, \dots , a_{n-1,1},\dots, a_{n-1,d}\} = \{b_1,\dots,b_{nd}\}$. Moreover let $c=a_{0,1}\cdots \hat{a}_{0,\eta(0)}\cdots a_{0,d}\cdots a_{n-1,1}\cdots \hat{a}_{n-1,\eta(n-1)}\cdots a_{n-1,d}$. Clearly, if $\nu\neq \eta$, we have 
    \[
    c\cdot a_{0,\nu(0)}\cdots a_{n-1,\nu(n-1)}=0.
    \]
    Thus $c\in \bigcap_{\nu \neq \eta} H(a_{0,\nu(0)},\dots, a_{n-1,\nu(n-1)}) = \bigcap_\nu H(a_{0,\nu(0)},\dots, a_{n-1,\nu(n-1)})$. In particular we have 
    \[
    c\cdot  a_{0,\eta(0)}\cdots a_{n-1,\eta(n-1)}=0
    \]
    but $c\cdot a_{0,\eta(0)}\cdots a_{n-1,\eta(n-1)}=b_1\cdots b_{nd}$. 
    Thus $S(R)$ is nilpotent (of class $\leq nd$). By Lemma \ref{lm:inductionStepNilp}, it is enough to conclude.
\end{proof}

In the NIP case, once we believed we had to show the nilpotency of $S(R)$, we firstly considered indiscernible sequences and showed they were nilpotent (of bounded index). Once we knew that, the Ramsey theorem was used to conclude because adding unrelated elements in between elements of the indiscernible sequence didn't change much. More precisely it would be fair to say the proof of theorem \ref{thm:secondorderelementsinBS} is nothing but the combination of the remark on indiscernible sequences of second order elements in BS$_0$ rings and the following lemma:
\begin{lemma} \label{lm:RamseyInductionIndiscernible}
    Let $M$ a semigroup and let $X$ be a definable of $M$. Assume there is a $k$ such that for every indiscernible sequences $(a_i)_{i<\omega}$ of elements in $X$, we have $a_1M^1a_2M^1\cdots a_{k-1}M^1a_k = 0$. 
    Then there is an integer $r$ such that $a_1\cdots a_r=0$ for every $a_1,\dots,a_r\in X$.
\end{lemma}

\begin{proof}
    Let $\pi((x_i)_{i<\omega})$ be the partial type defined by \[\pi(\overline{x})=\{\varphi(x_{i_1},\dots,x_{i_l})\iff \varphi(x_{j_1},\dots,x_{j_l})\mid \varphi\in L,   i_1<\cdots <i_l, j_1<\cdots<j_l\} \]
    Since every realization of $\pi$ is an indiscernible sequence, we deduce that $\pi\cup\{X\}\cup\{ x_1\cdots x_r\neq 0\}$ is inconsistent. \\
    Thus by compactness there is a finite number $\varphi_1,\dots,\varphi_k$ of formulas such that \[
    M\models \bigwedge_{l\leq k}\big(\varphi(a_{i^{(l)}_1},\dots,a_{{i^{(l)}_{m_l}}})\iff \varphi(a_{j_1^{(l)}},\dots,a_{{j^{(l)}_{m_l}}})\big) \text{ implies }  a_1M^1a_2M^1\cdots a_{k-1}M^1a_k = 0
    \]
    
    $\blacktriangleright$ Let $F_l:[\N]^{m_l}\to 2$ such that $M\models \varphi(a_{i_1},\dots,a_{m_l})\iff F(a_{i_1},\dots,a_{m_l})=1$. \\
    By Ramsey theorem, we can find an infinite subset $H\subseteq \N$ homogeneous for every $F_l$. In particular, by earlier assumption, for every $i_1<\cdots<i_k\in H$, we have  $a_{i_1}M^1a_{i_2}M^1\cdots a_{i_{k-1}}M^1a_{i_k}=0$. Thus if we note $H=\{h_i\mid i<\omega\}$ with $(h_i)$ increasing, we have $a_1\cdots a_{h_k}=0$. Indeed, $a_1\cdots a_{h_k}=(a_1\cdots a_{h_1-1}) \cdot a_{h_1}\cdot (a_{h_1+1}\cdots a_{h_2-1}) \cdot a_{h_2}\cdot (a_{h_2+1}\cdots a_{h_3-1}) \cdot a_{h_3}  \cdots a_{h_k} = 0$.
\end{proof}
In the following, we will try to adapt the rest of the proof of the NIP case to the $n$-NIP case. For that we will use the following family of subgroups:
\[
H(a_1,\dots,a_n)=\operatorname{RAnn}(R^1a_{\sigma(1)}R^1 a_{\sigma(2)}R^1\cdots a_{\sigma(n)})=\{x\in R\mid \forall \sigma\in \mathfrak{S}_n, \forall \overline{y}\in R\cup\{1\}, xy_1a_{\sigma(1)}y_2\cdots y_na_{\sigma(n)}=0\}
\]
Applying the $n$-NIP chain condition to this family of subgroup, and reasoning like the as in the lemma \ref{fact:commutativekNIPNilp}, we find the following lemma:
\begin{lemma}
    Let $R$ be a $n$-NIP ring. There is an integer $r$ such that for every $a_1,\dots,a_{nr}\in R$ second order elements, there is $i_1<\cdots <i_n$, such that for every $\sigma\in \mathfrak{S}_{nr}$ fixing the set $\{i_1,\dots,i_n\}$:
    \[
    a_{\sigma(1)}R^1a_{\sigma(2)}R^1\cdots a_{\sigma(i_1-1)}R^1  a_{\sigma(i_1+1)}R^1  \cdots  a_{\sigma(i_2-1)}R^1  a_{\sigma(i_2+1)}R^1  \cdots a_{\sigma(nr)} \cdot R^1 a_{\sigma(i_1)}R^1a_{\sigma(i_2)}R^1\cdots a_{\sigma(i_n)}=0.
    \]
    In other words, if $u$ is a product of terms containing the $a_j$ for $j\in\llbracket 1,nr\rrbracket\setminus \{i_1,\dots,i_n\}$, and $v$ is a product of terms containing the $a_j$ for $j\in  \{i_1,\dots,i_n\}$ then we have $u\cdot v=0$.
\end{lemma}
To understand how this lemma can be applied for the nilpotency of indiscernible sequences, let's deal with the $2$-NIP case first:
\begin{lemma}
    Let $R$ be a $2$-NIP ring and let $(a_i)_{i<\omega}$ be an indiscernible sequence of second order elements. \\
    Then there is an integer $r$ such that for every $b_l\in a_{lr+1}R^1a_{lr+2}R^1 \cdots a_{lr+r-1}R^1a_{lr+r}$, we have $b_3 b_0R^1b_2b_1=0$ i.e. for all $j_0<j_1<j_2<j_3$, we have $b_{j_3} b_{j_0}R^1b_{j_1}b_{j_2}=0$.
\end{lemma}
To simplify the notation, we will only prove the lemma for $b_l=a_{lr+1}a_{lr+2}\cdots a_{lr+r}$, for the general case is exactly the same (although notationally more cumbersome).
\begin{proof}
    Let $r$ be the integer defined last lemma. Let $1\leq i_1\leq r<i_2$ be integers such that $u$ is a product of terms containing the $a_j$ for $j\in\llbracket 1,2r\rrbracket\setminus \{i_1,i_2\}$, and $v$ is a product of terms containing $a_{i_1}$ and $a_{i_2}$ then $u\cdot v=0$. \\
    Applying this identity to the sequence $(a_1,\dots,a_{i_1-1},a_{r+1},a_{3r+i_1+1-i_2},\dots, a_{3r},a_{2r+i_2},\dots,a_{4r})$, whose $i_1$th term is $a_{r+1}$ and whose $i_2$th term is $a_{3r}$ and we get that \[a_{2r+i_2}\cdots a_{4r}a_1\cdots a_{i_1-1}R^1  a_{3r+i_1+1-i_2}\cdots a_{3r-1}\cdot a_{3r}a_{r+1} =0\] and thus $a_{3r+1}\cdots a_{4r}a_1\cdots a_{r}R^1  a_{2r+1}\cdots  a_{3r}a_{r+1}\dots a_{2r} =0$ and thus finally $b_3 b_0R^1b_2b_1=0$.
\end{proof}
\begin{corollary}
    Let $R$ be a $2$-NIP ring verifying an identity of the form $x_1\cdots x_u+\sum_{\sigma\neq \operatorname{Id}} \lambda_{\sigma}x_{\sigma(1)}\cdots x_{\sigma(u)}=0$. \\
    There is an integer $m$ such that for every $(a_i)_{i<\omega}$ be an indiscernible sequence of second order elements, we have $a_1R^1 a_2 R^1\cdots a_{m-1} R^1 a_m=0$.
\end{corollary}
\begin{proof}
    Let $b_l\in a_{lr+1}R^1a_{lr+2}R^1 \cdots a_{lr+r-1}R^1a_{lr+r}$ as in the previous theorem. 
    Notice that we have $b_{u+1} b_0\cdot\big (b_1\cdots b_u+\sum_{\sigma\neq \operatorname{Id}} \lambda_{\sigma}b_{\sigma(1)}\cdots b_{\sigma(u)})=0$, but $b_{u+1} b_0\cdot b_{\sigma(1)}\cdots b_{\sigma(u)}=0$ if $\sigma \neq \operatorname{Id}$. \\
    Thus $b_{u+1} b_0\cdots b_u=0$. Thus with $c_i = b_{(u+1)i}\cdots b_{(u+1)i+u}$, we have $c_i\cdot c_j=0$ if $i\geq j$. \\
    Thus reapplying the identity to the $c_i$ we get that $c_1\cdots c_u=0$ thus $a_1 R^1\cdots R^1a_{u^2+ur}=0$.
\end{proof}
In particular, using the lemma on indiscernible sequences of second order elements (lemma \ref{lm:RamseyInductionIndiscernible}) and the induction lemma (lemma \ref{lm:inductionStepNilp}), we obtain:

\begin{theorem}
    $2$-NIP rings of finite nilexponent are nilpotent.
\end{theorem}

\section{Engel Lie algebras}

\subsection{Enveloping algebra}
We now turn to NIP $n$-Engel Lie algebras. Our analysis consists in understanding the enveloping algebra, and describing a weaker form of the Baldwin-Saxl condition that will transfer from the Lie algebra to its enveloping algebra.

Given a Lie algebra $L$, the \textit{enveloping algebra} $A(L)$ of $L$ is the subalgebra of the linear endomorphisms algebra $\End(L)$ generated by all adjoint maps $\ad(a):x\mapsto [x,a]$. As we only care for nilpotency results, we may always assume that the Lie algebra is centerless, in which case it embeds into it enveloping algebra. Therefore, any element $a\in L$ will be naturally identified with its adjoint $ad(a)\in A$ and we may assume that $L\seq A$. The centralizer $C(a)$ of an element $a\in L$ is given by $\set{b\in L\mid [a,b] = 0}$.

Following the previous sections, we will say a Lie algebra $L$ is BS$_0$ if the poset of centralizers of finite subset of $L$ is of finite breadth. NIP Lie algebra are BS$_0$. However, it might be that the enveloping algebra of a BS$_0$ Lie algebra is not BS$_0$ itself, therefore, we need a refinement of our notion.

\begin{definition}
A semigroup $S$ is called \textit{BS$_0$} over $A\subseteq S$ if the $\wedge$-semilattice generated by $\{\operatorname{RAnn}_S(a)\mid a\in A\}$ is of finite breadth. In other words, there exists some $k\in \N$ such that for all $a_1,\dots,a_n\in A$ with $n\geq k$, there is $I\subseteq \llbracket 1,n\rrbracket$ of cardinality $\leq k$ such that $\operatorname{RAnn}(a_1,\dots,a_n)=\bigcap_{i\in I} \operatorname{RAnn}(a_i)$.
\end{definition}



\begin{remark}
    While the notion of BS$_0$ is symmetric, being BS$_0$ over a subset is not a symmetric notion: \\
    Let $R = \langle (X_i)_{i<\omega}, (Y_j)_{j<\omega}\mid X_iY_j=0, Y_iX_i=0\rangle$. Then $R$ is right BS$_0$ over $ S=\{X_i\mid i<\omega\}$ but not left BS$_0$ over $S$. 
\end{remark}

\begin{lemma}
    Suppose $L$ is a BS$_0$ Lie algebra. Then $A(L)$ is BS$_0$ over $\operatorname{ad}(L)$.
\end{lemma}
\begin{proof}
    For any $(a_i)_{i\in I}$ in $L$ and $J\seq I$, we have $\bigcap_{i\in I}\operatorname{C}_L(a_i)=\bigcap_{j\in J}\operatorname{C}_L(a_j)$ implies $\operatorname{RAnn}(a_i\mid i\in I)=\operatorname{RAnn}(a_j\mid j\in J)$. Indeed, if $f\in \operatorname{RAnn}(a_i\mid i\in I)$, then for every $x\in L$, we have $\operatorname{ad}(a_j)\circ f(x)=[a_j, f(x)]=0$. Thus $f\in \operatorname{RAnn}(a_j\mid j\in J)$.
\end{proof}
\begin{remark}
    A slight generalization of this proof shows that if $B$ is a non-associative algebra which is BS$_0$ over $Y\subseteq B$, then $\operatorname{End}(B)$ is BS$_0$ over $\{x\mapsto y\cdot x\mid y\in Y\}$.
\end{remark}
But this weak transfer of BS$_0$ is more that enough: our result on indiscernibles in BS$_0$ semigroup only use relative BS$_0$. Indeed we only consider Baldwin-Saxl on annihilator of the elements in the sequences of second order elements. This can be phrased as in the following lemma:

\begin{lemma}\label{lm:transferofBS0enveloping}
        Let $M$ be a semigroup and let $(a_i)_{i<\omega}$ be a sequence of second order elements. Assume $M$ is BS$_0$ over $\{a_i\mid i<\omega\}$. Then there is a sequence $(j_i)_{i<\omega}$ such that $j_0=1$ and by letting $b_l=a_{j_{l-1}}\cdots a_{j_l}$, we have $b_ib_j=0$ for all $i\geq j$.
\end{lemma}

Our results on NIP rings thus generalize if we consider the second order elements of $A(L)$: 

\begin{definition}
An element $a\in L$ is a \textit{toastie} if it generates an abelian ideal.
\end{definition}

Using Corollary \ref{cor:BS0PIalgebrassecondorder} and Lemma \ref{lm:transferofBS0enveloping}, we conclude the following.

\begin{proposition}\label{prop:enveloptransfer}
    Let $L$ be a BS$_0$ Lie algebra such that $A(L)$ satisfies a PI. Then there exists $k\in \N$ such that $a_1\cdot \cdots \cdot a_k = 0$ for all toasties $a_1,\ldots,a_k$ of $L$. 
\end{proposition}

\begin{remark}
More generally if $B$ is a BS$_0$ non-associative algebra whose enveloping algebra is PI, then by calling an element $b\in B$ a toastie if $R_b:x\mapsto xb$ and $L_b:x\mapsto bx$ are second order in $A(B)$, then we have an integer $k$ such that any product containing $k$ toasties is zero (no matter the bracketing !)
\end{remark}

\subsection{Countably categorical NIP Lie algebras} Let us first make some definitions.
\begin{definition}
The following definitions come from \cite{delbee2024engel3}.
\begin{itemize}
    \item A Lie algebra $L$ is \textit{$k$-strong} if for all $a\in L$, the ideal generated by $a$ is nilpotent of class $<k$.
    \item A Lie algebra is \textit{toastie} if it is generated by toastie elements.
\end{itemize}
\end{definition}

Note that if $L$ is a toastie Lie algebra, then $L$ is the span of the set of all toastie elements. The following fact of Zelmanov appears in \cite[p. 577]{Zelmanov95Lieringmethods}, and was later proved by Vaughan-Lee and Zelmanov \cite{VaughanLeeZelmanov1999}.

\begin{fact}(Vaughan-Lee, Zelmanov \cite[Proposition 1]{VaughanLeeZelmanov1999})
\label{claim:zelmanov}
    Let $L$ be an $n$-Engel Lie algebra of characteristic $p>n$. If $L$ is $k$-strong, then the enveloping algebra of $L$ is PI. 
\end{fact}

The following is a classical theorem of Kostrikin, it also yields a solution of the restricted Burnside problem for groups of prime exponent, see \cite{kostrikinAroundBurnside}.
\begin{fact}[Kostrikin, \cite{kostrikinAroundBurnside}]\label{fact:kostrikin}
    Let $L$ be an $n$-Engel Lie of characteristic $p>n$. Then $L$ contains a nontrivial toastie.
\end{fact}

The following classical fact will also be used.

\begin{fact}\label{fact:higgins}(Higgins, \cite{higginsEngel})
A solvable $n$-Engel Lie algebra of characteristic $p>n$ is nilpotent.
\end{fact}


\begin{theorem}\label{thm:kstrongtoastieomegaNIP}
    Let $L$ be an NIP $n$-Engel $k$-strong and toastie Lie algebra of characteristic $p>n$, then $L$ is nilpotent.
\end{theorem}
\begin{proof}
    By Claim \ref{claim:zelmanov}, the enveloping algebra of $L$ is PI.
    By Proposition \ref{prop:enveloptransfer}, there is $r\in \N$ such that $a_1\cdots a_r=0$ whenever $a_1,\dots,a_r\in L$ are toasties. As $L$ is generated by toasties, $L$ is nilpotent.
\end{proof}

\begin{corollary}\label{cor:NIPomegannilpote}
    Any $\omega$-categorical NIP $n$-Engel Lie algebra of characteristic $p>n$ is nilpotent.
\end{corollary}
\begin{proof}
    By $\omega$-categoricity, $L$ admits only a finite number of $\Aut(L)$-invariant ideals. Let 
    \[L = I_1\supsetneq \ldots \supsetneq I_{s+1} = 0\]
    be a maximal chain. Then each quotient $L_i = I_i/I_{i+1}$ is again $n$-Engel of characteristic $p>n$, $\omega$-categorical and NIP. Using Fact \ref{fact:higgins}, the corollary follows from the following claim. 
    \begin{claim}\label{claim:charsimplenilp}
        $L_i$ is nilpotent for each $i = 1,\ldots, s+1$.
    \end{claim}
    Indeed, under the claim, we have that $I_s$ and $L_{s-1} = I_{s-1}/I_s$ are nilpotent, therefore $I_{s-1}$ is solvable. By Fact \ref{fact:higgins}, it is also nilpotent. An immediate iteration yields that $L = I_1$ is nilpotent. 
    
    It remains to prove Claim \ref{claim:charsimplenilp}. Fix $i$ and set $M = L_i$. It is immediate to see that $M$ has no $\Aut(M)$-invariant ideal. By Kostrikin theorem (Fact \ref{fact:kostrikin}), the ideal generated by all toasties of $M$ is nontrivial, it is clearly $\Aut(M)$-invariant and therefore it must be $M$ itself. So $M$ is a toasties Lie algebra. By $\omega$-categoricity, there exists a number $k$ such that for every $a\in M$, there exists toasties $c_1,\ldots,c_{k-1}$ such that $a = c_1+\ldots+c_{k-1}$. Then the ideal of $M$ generated by $a$ is nilpotent of class $<k$. Therefore $M$ is $k$-strong hence nilpotent, by Theorem \ref{thm:kstrongtoastieomegaNIP}.
\end{proof}

\begin{remark}
    A conjecture of Zelmanov \cite[Conjecture (c)]{Zelmanov95Lieringmethods} states that for every $n$-Engel Lie algebra $L$ of characteristic $p>n$, the enveloping algebra $A(L)$ is PI. Under this conjecture, Corollary \ref{cor:NIPomegannilpote} follows immediately from Proposition \ref{prop:enveloptransfer} and without the $\omega$-categoricity assumption.
\end{remark}

\bibliographystyle{plain}
\bibliography{biblio.bib}{}

\end{document}